\documentclass{elsarticle}
\usepackage[a4paper,margin=1in]{geometry}
\usepackage{xcolor,soul,cancel}
\ifdefined\ManuscriptShowKeys\usepackage[color]{showkeys}\fi
\usepackage{amssymb,amsmath,amsthm,mathtools}
\usepackage{graphicx,float,enumitem,booktabs,hyperref}
\hypersetup{hidelinks}
\newlist{alphalist}{enumerate}{1} 
\setlist[alphalist]{label=(\alph*)} 
\numberwithin{equation}{section}

\newtheorem{thm}{Theorem}[section]
\newtheorem{lem}[thm]{Lemma}
\newtheorem{cor}[thm]{Corollary}
\newtheorem{prop}[thm]{Proposition}
\newtheorem{rem}[thm]{Remark}
\newtheorem{ex}[thm]{Example}

\newcommand{\eq} [1] {\begin{equation}\label{#1}\quad}
\newcommand{\en} {\end{equation}}

\newcommand{\diag}{\operatorname{diag}}

\theoremstyle{definition}

\theoremstyle{definition}
\newtheorem{definition}{Definition}

\DeclareMathOperator{\spanop}{span}
\DeclareMathOperator{\dist}{dist}
\DeclareMathOperator{\Hess}{Hess}
\DeclareMathOperator{\Spec}{Spec}
\DeclareMathOperator{\Vol}{Vol}
\DeclareMathOperator{\Rea}{Re}
\title{Rayleigh-Residual Flow I: Normal Matrices}
\author[eric]{Eric Shen}
\ead{erick.2013@yandex.ru}
\address[eric]{Moscow State University, Moscow, 119991, Russia \\
Institute for Numerical Mathematics, Russian Academy of Sciences, Russia}
\date{}

\newcommand{\PaperBibliography}{references}
\providecommand{\PaperNormalMaxResidual}{1.89\times 10^{-8}}

\providecommand{\PaperDiskHaarCoverage}{15.0}
\providecommand{\PaperDiskHaarCertifiedCoverage}{0.0}

\providecommand{\PaperDiskHaarInitialRadius}{0.057}
\providecommand{\PaperDiskHaarFinalRadius}{0.243}
\providecommand{\PaperDiskHaarResidual}{0.168}
\providecommand{\PaperDiskOneCoverage}{64.2}
\providecommand{\PaperDiskOneCertifiedCoverage}{0.0}

\providecommand{\PaperDiskOneResidual}{0.158}
\providecommand{\PaperDiskFourCoverage}{77.5}
\providecommand{\PaperDiskFourCertifiedCoverage}{1.7}

\providecommand{\PaperDiskFourResidual}{0.150}

\begin{document}
\begin{frontmatter}
\begin{abstract}
We study the squared Rayleigh-residual functional on complex projective
space and its negative gradient flow. For an arbitrary complex matrix,
every positive-residual critical point has a strictly descending tangent
direction; consequently, the local minima are precisely the eigenlines.
For normal matrices, we describe the critical set in terms of circles
through spectral points and identify the amplitude dynamics with a
replicator equation whose payoff matrix has rank at most four. This gives
explicit logarithmic first integrals and a reduced representation of
interior trajectories. We compute transverse escape rates near critical
components and analyze how polynomial reweighting changes unstable
coordinates. Finally, we distinguish local saddle trapping from
concentration of Haar-distributed initial states and derive the density
of the filtered ensemble. Numerical examples with random spectra and a
Jordan block illustrate convergence, nonmonotone Rayleigh motion, and
the improvement of finite-time spectral exploration by low-degree
random polynomial filters.
\end{abstract}


\end{frontmatter}
\section{Introduction}\label{sec:introduction}

Given an arbitrary square matrix $A\in M_n(\mathbb C)$, we study the
functional $\rho_A$ on the unit sphere
$\mathbb S^{2n-1}\subset\mathbb C^n$ and, more importantly, its gradient
vector field and associated flow $\Phi_A^t$. The functional is given by
\[
    z_A(u)=u^*Au,\qquad
    R_A(u)=Au-z_A(u)u,\qquad
    \rho_A(u)=\|R_A(u)\|^2,
\]
that is, $\rho_A(u)$ is the squared Rayleigh residual. We show in
Theorem~\ref{thm:no_pos_res_min} that a unit vector $u$ is a local
minimum of $\rho_A$ if and only if it is an eigenvector of $A$.
Proposition~\ref{prop:gradient-flow} gives an explicit gradient formula,
so each Euler step requires only one matrix-vector product with $A$
and one with $A^*$.

These observations suggest a matrix-free procedure for randomized
parallel computation of eigenpairs of a normal matrix:
\begin{enumerate}
    \item Take a Haar-distributed unit vector $u_0$.
    \item While $\rho_A(u_k)>\varepsilon^2$, take retracted Euler steps
    of the negative gradient flow.
    \item Once the residual reaches this threshold, return $u_k$ and
    $z_A(u_k)$. For normal $A$, the latter lies within $\varepsilon$ of
    the spectrum; see Proposition~\ref{prop:gradient-flow}.
\end{enumerate}
Taking many independent starts at once produces an ensemble of
trajectories. Theorem~\ref{thm:no_pos_res_min} rules out attracting local
minima with positive residual. Positive-residual critical points still
exist, however, and trajectories starting on their stable manifolds
need not converge to eigenvectors. In addition, different eigenpairs
can have very different sampling probabilities, and trajectories can
spend a long finite time near the saddle set.

\par There are two classical viewpoints closely related to the structures appearing below. First, Burke, Lewis, and Overton \cite{BLO03} proved that the least-singular-value function

$$ z\longmapsto \sigma_{\min}(A-zI) $$

has no local minimizers away from the spectrum. Their problem and ours arise from opposite partial minimizations of the joint residual

$$ F(z,u)=\|(A-zI)u\|^2: \qquad \rho_A(u)=\min_{z\in\mathbb C}F(z,u), \qquad \sigma_{\min}(A-zI)^2=\min_{\|u\|=1}F(z,u). $$

 Second, in the normal case the squared pairwise eigenvalue distances form a Euclidean distance matrix. Classical results of Gower on Euclidean distance matrices (see \cite{GOWER198581}) underlie both the circumcenter geometry of the critical set (see Theorem~\ref{thm:normal-critical-set}) and the rank-four structure responsible for the first integrals in Section~\ref{sec:qual_prop}, see Proposition~\ref{prop:first-integrals} and Remark~\ref{rem:edm-rank} after it.

We study the qualitative properties of the normal flow in
Section~\ref{sec:qual_prop}. Theorem~\ref{thm:normal-critical-set}
characterizes its critical components by subsets of eigenvalues lying
on a common circle, and Theorem~\ref{thm:gen_crit_skel} describes the
generic pair and triple components. Section~\ref{sec:quan_prop} gives
the corresponding transverse escape rates and local trapping estimates;
see Theorem~\ref{thm:local-trapping}.
In Section~\ref{sec:critical-filtering} we analyze polynomial filters
as perturbations of the unstable coordinates, including an exact
pair-saddle calculation in Proposition~\ref{prop:pair-escape} and
Corollary~\ref{cor:pair-acceleration}.
Finally, Section~\ref{sec:concentration} studies concentration of the
initial ensemble and the change of sampling measure under polynomial
filtering. The numerical examples illustrate both the local convergence
mechanism and the separate question of spectral coverage.

\section{Residual flow}\label{sec:residual-flow}

In this section we define the aforementioned flow and discuss its basic properties.
\begin{definition} \label{def:bas_def}
    
Let $A\in M_n(\mathbb C)$. For a unit vector $u\in\mathbb C^n$, define
\[
    z_A(u)=u^*Au,
    \qquad
    R_A(u)=Au-z_A(u)u,
    \qquad
    \rho_A(u)=\|R_A(u)\|^2.
\]
\end{definition}
The most basic properties of the residual functional are collected in the following Proposition.
\begin{prop} \label{prop:basic_prop}
The following properties hold.

\begin{enumerate}
    \item \emph{Phase invariance.} For every \(\theta\in\mathbb R\),
    \[
        z_A(e^{i\theta}u)=z_A(u),\qquad
        R_A(e^{i\theta}u)=e^{i\theta}R_A(u), \qquad
        \rho_A(e^{i\theta}u)=\rho_A(u).
    \]
        
    In particular, \(\rho_A\) descends to a well-defined function on
    \(\mathbb CP^{n-1}\).

    \item \emph{Unitary equivariance.} If \(U\) is unitary, then
        $\rho_{U^*AU}(v)=\rho_A(Uv). $

    \item \emph{Affine covariance in the matrix variable.} For
    \(\alpha,\lambda\in\mathbb C\),
        $\rho_{\alpha A+\lambda I}(u)=|\alpha|^2\rho_A(u)$
\end{enumerate}

Consequently, the projective residual landscape is intrinsic under unitary
similarity: replacing \(A\) by \(U^*AU\) only changes coordinates on
\(\mathbb CP^{n-1}\). Moreover, if \(\Phi_A^t\) denotes the negative
gradient flow of \(\rho_A\) on projective space, with respect to the standard
Fubini--Study metric, then
\[
        \Phi_{U^*AU}^t([U^*u])
        =
        [\,U^*\Phi_A^t([u])\,],
\]
and
\[
        \Phi_{\alpha A+\lambda I}^t
        =
        \Phi_A^{|\alpha|^2 t}.
\]

\end{prop}

\begin{proof}
The phase identities follow directly from
\[
        (e^{i\theta}u)^*A(e^{i\theta}u)=u^*Au, \,\,\, \text{and}
\]
\[
        A(e^{i\theta}u)-z_A(u)e^{i\theta}u
        =
        e^{i\theta}(Au-z_A(u)u).
\]
For unitary equivariance, put \(v=U^*u\). Then
\[
        z_B(v)
        =
        v^*Bv
        =
        u^*U(U^*AU)U^*u
        =
        u^*Au
        =
        z_A(u).
\]
Therefore
\[
        R_B(v)
        =
        Bv-z_B(v)v
        =
        U^*AUU^*u-z_A(u)U^*u
        =
        U^*(Au-z_A(u)u)
        =
        U^*R_A(u),
\]
and taking norms gives \(\rho_B(v)=\rho_A(u)\).

Finally,
\[
        z_{\alpha A+\lambda I}(u)
        =
        u^*(\alpha A+\lambda I)u
        =
        \alpha z_A(u)+\lambda, \,\,\, \text{so}
\]
\[
        R_{\alpha A+\lambda I}(u)
        =
        (\alpha A+\lambda I)u-(\alpha z_A(u)+\lambda)u
        =
        \alpha(Au-z_A(u)u)
        =
        \alpha R_A(u).
\]
Thus
\[
        \rho_{\alpha A+\lambda I}(u)=|\alpha|^2\rho_A(u).
\]
The corresponding statements for the gradient flow follow from these
identities and from the unitary invariance of the Fubini--Study metric.
\end{proof}

Thus, when studying properties of the residual flow, one may replace the matrix by a unitarily similar one. The next theorem describes the local minima of the residual functional. It rules out positive-residual local minima, for normal and nonnormal matrices alike.
\begin{thm}\label{thm:no_pos_res_min}
Let $u\in\mathbb C^n$, $\|u\|=1$, be such that $[u]$ is a critical point of $\rho_A$ on $\mathbb{CP}^{n-1}$.
Then exactly one of the following alternatives holds.

\begin{enumerate}
    \item $\rho_A(u)=0$. In this case $R_A(u)=0$, so $u$ is an eigenvector of $A$.
    \item $\rho_A(u)>0$. In this case $u$ has a strictly descending projective tangent direction. More precisely, there exists a horizontal tangent vector $h\in\mathbb C^n$ satisfying
    \[
        u^*h=0,
        \qquad
        \|h\|=1,
    \]
    such that, with the convention
    \[
        \rho_A(u_h(t))
        =
        \rho_A(u)
        +\frac{t^2}{2}\,\Hess \rho_A(u)[h,h]
        +O(t^3),
    \]
    one has
    \[
        \Hess \rho_A(u)[h,h]
        \leq
        -2\rho_A(u).
    \]
\end{enumerate}
Consequently, no critical point with positive residual can be a local minimum of $\rho_A$ on $\mathbb{CP}^{n-1}$.
\end{thm}
We will use the following Lemma which gives equations for the critical point. Until the end of this section we fix the notation:
$\|u\|=1$, $z=u^*Au$, $B=A-zI$, and $s^2=\|Bu\|^2$.
\begin{lem}\label{lem:crit_point_equ}
A vector $u$ is critical for $\rho_A$ on $\mathbb{CP}^{n-1}$ if and only if  $B^*Bu=s^2u$.
If $s>0$ and $v=\frac{Bu}{s}$,
then
\[
    Bu=sv,
    \qquad
    B^*v=su,
    \qquad
    u^*v=0,
    \qquad
    \|v\|=1.
\]
\end{lem}

\begin{proof}
Take a horizontal tangent vector $h\in H_u$. Since $u^*Bu=0$,
$$
    \rho_A(u+th)=\rho_B(u+th)=\|B(u+th)\|^2-|(u+th)^*B(u+th)|^2. $$
The derivative of the first term is
$$
    \frac{d}{dt}\bigg|_{t=0}\|B(u+th)\|^2
    =
    2\operatorname{Re}\,(Bu)^*Bh
    =
    2\operatorname{Re}\,h^*B^*Bu.
$$
The derivative of the second term is zero, because since $u^*Bu=0$ we have
$$
    (u+th)^*B(u+th)=u^*Bu + (u^*Bh+h^*Bu)t + O(t^2)=O(t).
$$
Therefore
\[
    D\rho_A(u)[h]
    =
    2\operatorname{Re}\,h^*B^*Bu
    \qquad (h\in H_u).
\]
Thus $u$ is critical on projective space if and only if
\[
    \operatorname{Re}\,h^*B^*Bu=0
    \qquad
    \text{for all }h\in H_u.
\]
This is equivalent to $B^*Bu$ being collinear with $u$. Hence $B^*Bu=\mu u$ for some scalar $\mu$. Taking the inner product with $u$ gives
\[
    \mu
    =u^*B^*Bu
    =\|Bu\|^2
    =s^2.
\]
So the critical point equation is
\[
    B^*Bu=s^2u.
\]

Now assume $s>0$ and set $v=Bu/s$. Then $Bu=sv$ and $\|v\|=1$. Also
\[
    u^*v=\frac{u^*Bu}{s}=0.
\]
Finally,
\[
    B^*v
    =\frac{B^*Bu}{s}
    =\frac{s^2u}{s}
    =su.
\]
This proves the lemma.
\end{proof}
The proof of the Theorem is now straightforward: we invoke the identities from Lemma~\ref{lem:crit_point_equ} and show that the Hessian of $\rho_A(u)$ has a negative direction whenever $\rho_A(u)>0$. 
\begin{proof}[Proof of Theorem~\ref{thm:no_pos_res_min}]
    If $s=0$, then $\|Bu\|=0$, hence $Bu=0$. Since $B=A-zI$, this means
\[
    Au=zu,
\]
so $u$ is an eigenvector of $A$.

Now suppose $s>0$. By Lemma~\ref{lem:crit_point_equ},
\[
    v=\frac{Bu}{s}
\]
is well-defined and satisfies
\[
    \|v\|=1,
    \qquad
    u^*v=0,
    \qquad
    Bu=sv,
    \qquad
    B^*v=su.
\]
In particular, $u$ and $v$ are orthonormal. Define $W=\spanop\{u,v\}^{\perp}$ and observe that
\[
    \dim_{\mathbb C}H_u=n-1,
    \qquad
    \dim_{\mathbb C}W=n-2.
\]
Consider the linear map
\[
    T=P_WB|_{H_u}:H_u\longrightarrow W,
\]
where $P_W$ is the orthogonal projection onto $W$. The domain is larger than the codomain, so there is non-zero $h\in\ker T$.
Then $h\in H_u$, so $u^*h=0$, and $P_WBh=0$. Moreover,
\[
    v^*Bh=(B^*v)^*h=(su)^*h=su^*h=0.
\]
Thus $Bh$ has no $W$-component and no $v$-component. Since $\mathbb C^n=\spanop\{u\}\oplus\spanop\{v\}\oplus W$
is an orthogonal decomposition, we conclude that $Bh=\alpha u$ for some $\alpha\in\mathbb C$.
Now observe that $$
    \frac{1}{2}\Hess\rho_A(u)[h,h]=Q(h)     =
    \|Bh\|^2
    -s^2\|h\|^2
    -\left|u^*Bh+h^*Bu\right|^2. 
$$

Since $Bu=sv$ and $Bh=\alpha u$, we have
\[
    u^*Bh=\alpha,
    \qquad
    h^*Bu=s h^*v.
\]
Therefore
\[
    Q(h)
    =
    |\alpha|^2-s^2\|h\|^2-|\alpha+s h^*v|^2.
\]

We may replace $h$ by $e^{i\theta}h$. This preserves all constraints: it remains horizontal, remains in $\ker T$, and still satisfies $B(e^{i\theta}h)=e^{i\theta}\alpha u$. Under this replacement,
\[
    \alpha\mapsto e^{i\theta}\alpha,
    \qquad
    h^*v\mapsto e^{-i\theta}h^*v.
\]
Thus the modulus term becomes $\left|e^{i\theta}\alpha+s e^{-i\theta}h^*v\right|$. Choose $\theta$ so that the two summands have the same argument. Hence, for a suitable phase choice,
\[
    \left|e^{i\theta}\alpha+s e^{-i\theta}h^*v\right|
    =
    |\alpha|+s|h^*v|.
\]
For this phase-rotated vector, which we again denote by $h$, we get
\[
\begin{aligned}
Q(h)
&=|\alpha|^2-s^2\|h\|^2-\bigl(|\alpha|+s|h^*v|\bigr)^2 \leq -s^2\|h\|^2.
\end{aligned}
\]
Since $h\neq0$, normalize it so that $\|h\|=1$. Then $Q(h)\leq -s^2$. Finally,
\[
    \Hess\rho_A(u)[h,h]=2Q(h)\leq -2s^2=-2\rho_A(u).
\]
Thus every positive-residual critical point has a descending projective tangent direction. In particular, it cannot be a local minimum.

\end{proof}

Let $F(z,u)=\|(A-zI)u\|^2,\,\|u\|=1.$
Minimizing \(F\) in \(z\) gives
$$ \rho_A(u)=\min_z F(z,u), \qquad z=u^*Au, $$
whereas minimizing in \(u\) gives
$$ \sigma_{\min}(A-zI)^2=\min_{\|u\|=1}F(z,u). $$

Burke, Lewis, and Overton proved in \cite{BLO03} that \(z\mapsto\sigma_{\min}(A-zI)\) has no positive local minima. Thus their result is complementary to Theorem~\ref{thm:no_pos_res_min}. The two statements are not equivalent: at a critical point of \(\rho_A\),
$$ (A-zI)^*(A-zI)u=\rho_A(u)u, $$
so \(u\) belongs to a singular-value branch of \(A-zI\), but this singular value need not be the smallest one.

\section{Examples of the residual flow}
\begin{prop}
\label{prop:gradient-flow}
Put $z=z_A(u)$, $R=R_A(u)$, and $\rho=\rho_A(u)$. With the metric
normalization used here, the horizontal gradient is
\[
    g_A(u)=2\bigl(A^*R-\overline zR-\rho u\bigr).
\]
Thus the negative gradient flow and its retracted Euler discretization are
\[
    \dot u=-g_A(u),\qquad
    u_{k+1}=\frac{u_k-hg_A(u_k)}{\|u_k-hg_A(u_k)\|}.
\]
Along the exact flow, $\dot\rho_A=-\|g_A\|^2$.
Computing $g_A(u)$ requires one application of $A$ and one of $A^*$.
If $A$ is normal, then
\[
    \dist\bigl(z_A(u),\Spec(A)\bigr)\leq\sqrt{\rho_A(u)}.
\]
\end{prop}
\begin{proof}
For a horizontal tangent vector $h$, differentiation gives
\[
    D\rho_A(u)[h]
    =2\operatorname{Re}\bigl(h^*(A^*-\overline zI)R\bigr),
\]
because $u^*R=0$. Moreover,
$u^*(A^*-\overline zI)R=\rho$, so horizontal projection gives the
stated gradient. The decrease of $\rho_A$ follows from the gradient-flow
identity. To evaluate the formula, first compute $Au$, then $A^*R$;
all remaining operations are scalar products and vector updates.
For normal $A$, an orthonormal eigenbasis gives
$\rho_A(u)=\sum_i|u_i|^2|\lambda_i-z_A(u)|^2$.
This weighted average is at least
$\dist(z_A(u),\Spec(A))^2$.
\end{proof}
\begin{ex}
Let $A=\operatorname{diag}(\lambda_1,\ldots,\lambda_{60})$, where the
$\lambda_i$ are sampled independently and uniformly by area from the unit
disk. We evolve $32$ independent Haar-distributed unit vectors by retracted
Euler steps of size $h=0.05$, without filtering or early stopping.
Figure~\ref{fig:paper-random-normal} shows the initial Rayleigh points,
their subsequent paths, and the decrease of the residuals.
At $t=5000$, the largest residual norm among these trajectories is
$\PaperNormalMaxResidual$; every final vector has more than $99.99\%$ of
its squared norm in a single eigendirection.
\end{ex}

\begin{figure}[H]
  \centering
  \includegraphics[width=\linewidth]{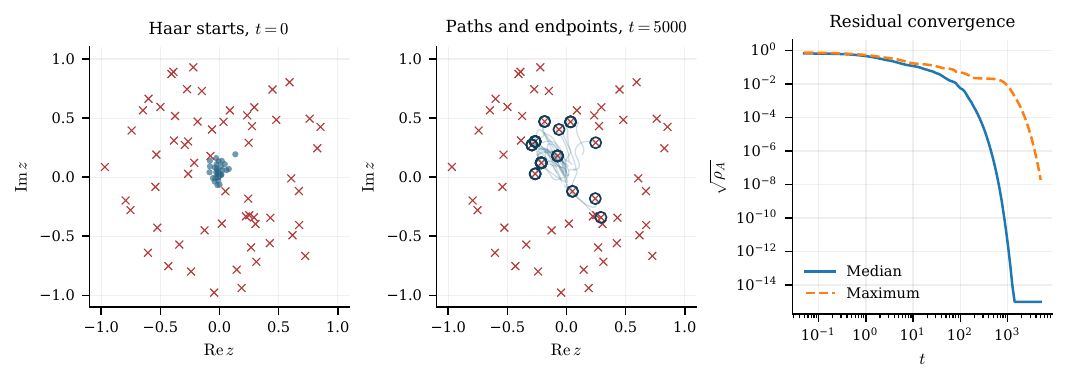}
  \caption{Residual flow for a normal matrix with $60$ independently sampled
  disk eigenvalues. Left: the $32$ initial Haar Rayleigh points. Middle:
  their paths and final points at $t=5000$, shown as open circles; red crosses
  indicate eigenvalues. Right: median and maximum residual norms. The
  trajectories use the same unfiltered retracted Euler flow throughout.}
  \label{fig:paper-random-normal}
\end{figure}

\begin{ex}

For the Jordan block
\[
  J=\begin{pmatrix}0&1\\0&0\end{pmatrix},
\]
write $p=|u_2|^2$. Direct calculation gives
\[
  z_J(u)=\overline{u_1}u_2,\qquad
  |z_J(u)|^2=p(1-p),\qquad \rho_J(u)=p^2.
\]
With the gradient convention used here, the exact flow satisfies
\[
  \dot p=-8p^2(1-p).
\]
For every $0<p(0)<1$, the variable $p(t)$ decreases to zero and $[u(t)]$
converges to the eigenline $[e_1]$. If $p(0)>1/2$, however, $|z_J(u(t))|$
first increases until $p=1/2$, then decreases to zero. Hence the distance
of the Rayleigh point to the spectrum need not decrease monotonically, even
though the residual does. The exceptional point $[e_2]$, with $p=1$, is
a positive-residual critical point and is stationary.
\end{ex}
Figure~\ref{fig:paper-jordan-return} shows both the excursion and the return.
For the numerical run, we use $600$ Haar starts and four prescribed starts
with $p(0)=0.1,0.9,0.99,0.999$, with Euler step $h=0.005$.
At $t=10^3$, all plotted trajectories satisfy
$|z_J|<1.2\times10^{-2}$ and
$\sqrt{\rho_J}<1.3\times10^{-4}$.
The return is algebraic: the exact scalar equation gives
$p(t)\sim(8t)^{-1}$ and $|z_J(u(t))|\sim(8t)^{-1/2}$.

\begin{figure}[H]
  \centering
  \includegraphics[width=\linewidth]{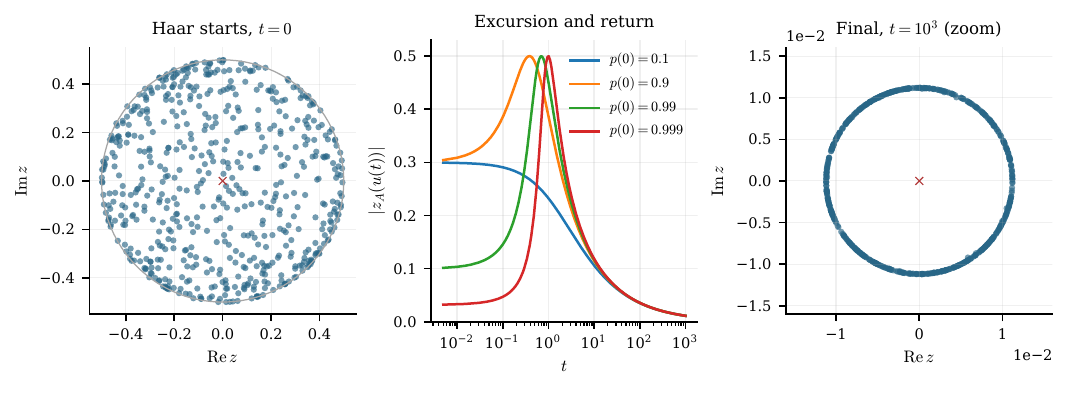}
  \caption{Excursion and return for the $2\times2$ Jordan block. Left:
  Rayleigh points of $600$ Haar starts, with the boundary of the numerical
  range shown in grey. Middle: distances to the unique eigenvalue $0$ for
  four separately prescribed initial weights. The three starts with
  $p(0)>1/2$ move outward before returning. Right: the Haar cloud at
  $t=10^3$, zoomed in.}
  \label{fig:paper-jordan-return}
\end{figure}

\section{Qualitative properties of the flow for normal matrices}\label{sec:qual_prop}
From now on we assume that $A \in M_n$ is normal. In this section, we describe explicitly the critical set of $\rho_A$. We show that in this case the residual landscape
is controlled by the elementary geometry of the eigenvalues in the complex
plane.
Since $A$ is normal and the properties of the flow are unitarily equivariant by Proposition~\ref{prop:basic_prop}, after a unitary change of basis we may assume
\[
    A=\operatorname{diag}(\lambda_1,\ldots,\lambda_n),
\]
 Denote, as before,
$$
    u=(u_1,\ldots,u_n)^T,
    \qquad
    p_i=|u_i|^2 \quad \text{and}
$$
$$
    p=(p_1,\ldots,p_n)\in\Delta^{n-1}
    :=
    \left\{p_i\geq 0,\ \sum_{i=1}^n p_i=1\right\},$$

% \begin{prop}
% Let $A=\operatorname{diag}(\lambda_1,\ldots,\lambda_n)$ and $z=\sum_{i=1}^n p_i\lambda_i$. Then
% $$
%     \rho_A(u)
%     =
%     \sum_{i=1}^n p_i|\lambda_i|^2-|z|^2
%     =
%     \frac12\sum_{i,j=1}^n p_ip_j|\lambda_i-\lambda_j|^2.
% $$
% In particular, $\rho_A$ depends only on the probability vector $p$, not on the phases of the coordinates of $u$.
% \end{prop}

% \begin{proof}
% The second equality is at hand from definition of $z$, so we need to prove only the first one. For diagonal $A$ we have
%     $Au-zu=((\lambda_1-z)u_1,\ldots,(\lambda_n-z)u_n)^T.$
% Therefore
% $$
%     \rho_A(u)=\|Au-zu\|^2=\sum_{i=1}^n |u_i|^2|\lambda_i-z|^2
%     =\sum_{i=1}^n p_i|\lambda_i-z|^2.
% $$
% Expanding gives
% $$
%     \sum_i p_i|\lambda_i-z|^2
%     =
%     \sum_i p_i\bigl(|\lambda_i|^2-2\operatorname{Re}(\lambda_i\overline z)+|z|^2\bigr)
%     =$$
    
%     $$
%     \sum_i p_i|\lambda_i|^2-2\operatorname{Re}\left(\overline z\sum_i p_i\lambda_i\right)+|z|^2\sum_i p_i
%     =
%     \sum_i p_i|\lambda_i|^2-|z|^2.
% $$
% \end{proof}
\begin{prop}\label{prop:normal-simplex-flow}
For the functional $\rho_A(u)$ the negative gradient flow on the probability simplex is
$$
    \boxed{
    \dot p_i=4p_i\left(\rho(p)-|\lambda_i-z(p)|^2\right).
    }
$$
\end{prop}

\begin{proof}
For diagonal $A$ we have
\[
    (A^*-\overline zI)(A-zI)u
    =
    \bigl(|\lambda_1-z|^2u_1,\ldots,|\lambda_n-z|^2u_n\bigr)^T.
\]
Therefore
\[
    \dot u_i
    =
    -2\left(|\lambda_i-z|^2-\rho\right)u_i
    =
    2\left(\rho-|\lambda_i-z|^2\right)u_i.
\]
 Also,
\[
    \dot p_i
    =
    \frac{d}{dt}|u_i|^2
    =
    2\Rea(\overline{u_i}\dot u_i)
    =
    4p_i\left(\rho-|\lambda_i-z|^2\right).
\]
This proves the formula.
\end{proof}
The amplitude dynamics of the normal residual flow belongs to the classical family of replicator equations. We briefly record this identification because, in the present Euclidean-distance case, the payoff matrix has very low rank and the flow acquires a large family of explicit first integrals. For the general theory of replicator dynamics we refer to Taylor--Jonker~\cite{TaylorJonker}, Shahshahani~\cite{Shahshahani}, and the monographs of Hofbauer--Sigmund~\cite{HofbauerSigmund} and Sandholm~\cite{Sandholm}. In particular, replicator dynamics for a potential game is the gradient flow of its potential with respect to the Shahshahani metric; see~\cite{Shahshahani,Sandholm}.

Put
$$
d_i=|\lambda_i-z|^2,\qquad
D_{ij}=|\lambda_i-\lambda_j|^2,
\qquad
B=-D.
$$
A direct calculation gives
$$
(Dp)_i=d_i+\rho,
\qquad
p^TDp=2\rho.
$$
Hence the residual flow can be written as
$$
\dot p_i
=4p_i\bigl((Bp)_i-p^TBp\bigr).
$$
Thus, up to the inessential factor $4$, it is the replicator equation of the symmetric matrix game with payoff matrix $B=-D$. Its potential is
$$
\frac12p^TBp=-\rho,
$$
so the decrease of the residual is the usual potential monotonicity of a symmetric replicator system, specialized to the squared Euclidean distance matrix of the spectral points.

\begin{prop}\label{prop:first-integrals}
Assume that the four real vectors
$$
(1)_{i=1}^n,
\qquad
(\operatorname{Re}\lambda_i)_{i=1}^n,
\qquad
(\operatorname{Im}\lambda_i)_{i=1}^n,
\qquad
(|\lambda_i|^2)_{i=1}^n
$$
are linearly independent. If $c=(c_1,\ldots,c_n)\in\mathbb R^n$ satisfies
$$
\sum_i c_i=0,
\qquad
\sum_i c_i\lambda_i=0,
\qquad
\sum_i c_i|\lambda_i|^2=0,
$$
then, on the interior of the simplex,
$$
I_c(p)=\sum_i c_i\log p_i
$$
is constant along the residual flow. Equivalently,
$$
\prod_i p_i^{c_i}=\mathrm{const}.
$$
The space of such first integrals has dimension $n-4$.
\end{prop}

\begin{proof}
Since
$$
\frac{\dot p_i}{p_i}=4(\rho-d_i),
$$
we obtain
$$
\frac{d}{dt}I_c(p)
=4\rho\sum_i c_i-4\sum_i c_i|\lambda_i-z|^2.
$$
Expanding the last term gives
$$
\sum_i c_i|\lambda_i-z|^2
=
\sum_i c_i|\lambda_i|^2
-2\operatorname{Re}\left(\overline z\sum_i c_i\lambda_i\right)
+|z|^2\sum_i c_i,
$$
which vanishes under the stated assumptions. Hence $\frac{d}{dt}I_c=0$.

The three displayed constraints consist of four independent real linear equations under the genericity assumption, so their solution space has dimension $n-4$.
\end{proof}

\begin{rem}\label{rem:edm-rank}
If
$$
M=
\begin{pmatrix}
1&\cdots&1\\
\operatorname{Re}\lambda_1&\cdots&\operatorname{Re}\lambda_n\\
\operatorname{Im}\lambda_1&\cdots&\operatorname{Im}\lambda_n\\
|\lambda_1|^2&\cdots&|\lambda_n|^2
\end{pmatrix},
$$
then $D=M^TSM$ with
$$
S=
\begin{pmatrix}
0&0&0&1\\
0&-2&0&0\\
0&0&-2&0\\
1&0&0&0
\end{pmatrix}.
$$
Consequently, under the hypothesis of Proposition~\ref{prop:first-integrals}, $\operatorname{rank}D=4$ and the coefficient vectors $c$ above are precisely $\ker D$. Thus the large number of first integrals is a direct consequence of the rank-four structure of a planar squared Euclidean distance matrix. If the spectrum is cocircular or collinear, the rank drops and additional logarithmic first integrals appear.
\end{rem}

The preceding integrals can be complemented by an exact representation of every interior trajectory. It makes explicit that, after fixing the $n-4$ constants from Proposition~\ref{prop:first-integrals}, the amplitude motion is only three-dimensional.

\begin{prop}\label{prop:exact-representation}
Let $p(t)$ be an interior trajectory and put
$$
W(t)=\int_0^t z(s)\,ds.
$$
Then
$$
p_i(t)=
\frac{
 p_i(0)\exp\left(-4t|\lambda_i|^2+8\operatorname{Re}(\overline{\lambda_i}W(t))\right)
}{
 \displaystyle\sum_j p_j(0)\exp\left(-4t|\lambda_j|^2+8\operatorname{Re}(\overline{\lambda_j}W(t))\right)
}.
$$
Equivalently, $W$ is the solution of the single complex nonautonomous equation
$$
\dot W=
\frac{
\displaystyle\sum_i \lambda_i p_i(0)
\exp\left(-4t|\lambda_i|^2+8\operatorname{Re}(\overline{\lambda_i}W)\right)
}{
\displaystyle\sum_i p_i(0)
\exp\left(-4t|\lambda_i|^2+8\operatorname{Re}(\overline{\lambda_i}W)\right)
},
\qquad
W(0)=0.
$$
Thus a generic trajectory in the $(n-1)$-dimensional simplex is contained in a three-real-dimensional exponential family parametrized by $t$, $\operatorname{Re}W$, and $\operatorname{Im}W$.
\end{prop}

\begin{proof}
Expanding $d_i=|\lambda_i-z|^2$ gives
$$
\frac{d}{dt}\log p_i
=
-4|\lambda_i|^2
+8\operatorname{Re}(\overline{\lambda_i}z)
+\beta(t),
$$
where
$$
\beta(t)=4\rho-4|z|^2
$$
is independent of $i$. Integration yields
$$
p_i(t)=C(t)p_i(0)
\exp\left(-4t|\lambda_i|^2+8\operatorname{Re}(\overline{\lambda_i}W(t))\right).
$$
The normalization $\sum_i p_i(t)=1$ determines $C(t)$ and gives the first formula. Substituting it into $\dot W=z=\sum_i p_i\lambda_i$ gives the reduced equation.
\end{proof}
Theorem~\ref{thm:normal-critical-set} characterizes critical values of $\rho_A$. It gives a clear geometric interpretation: a critical value corresponds to a certain set of eigenvalues lying on a circle. Before we proceed, we fix notation.

\begin{definition}
For a subset $S\subset\{1,\ldots,n\}$, write
\[
    \Delta_S^\circ
    =
    \left\{p\in\Delta^{n-1}: p_i>0\text{ for }i\in S,\ p_i=0\text{ for }i\notin S\right\}.
\]
The set $S$ is called the active support of $p$. 
\end{definition}
\begin{thm} \label{thm:normal-critical-set}
Let \(A\in M_n(\mathbb C)\) be normal, and write its spectral decomposition as
\[
        A=\sum_{\alpha=1}^d \mu_\alpha P_\alpha,
\]
where \(\mu_1,\ldots,\mu_d\) are the distinct eigenvalues of \(A\), \(P_\alpha\)
is the orthogonal projection onto the eigenspace $E_\alpha=\ker(A-\mu_\alpha I)$, and \(m_\alpha=\dim E_\alpha\).

For a unit vector \(u\), put $u_\alpha=P_\alpha u, \,
        p_\alpha=\|u_\alpha\|^2$,
so that $\sum_{\alpha=1}^d p_\alpha=1$.
Then
\[
        z(u)=u^*Au=\sum_{\alpha=1}^d p_\alpha \mu_\alpha, \qquad
        \rho(u)=\sum_{\alpha=1}^d p_\alpha |\mu_\alpha-z(u)|^2.
\]
Let   $S=\{\alpha:p_\alpha>0\}$. Then \([u]\in\mathbb CP^{n-1}\) is a critical
point of \(\rho\) if and only if
\[
        |\mu_\alpha-z(u)|^2=\rho(u)
        \qquad
        \text{for every }\alpha\in S.
\]
Equivalently, the eigenvalues  $\{\mu_\alpha:\alpha\in S\} $
lie on the circle centered at $z(u)$ with radius $\sqrt{\rho(u)}$.

% Conversely, suppose that \(S\subset\{1,\ldots,d\}\) and that
% \(p=(p_\alpha)_{\alpha\in S}\) is a positive probability vector such that, with
% \[
%         c=\sum_{\alpha\in S}p_\alpha\mu_\alpha,
% \]
% one has
% \[
%         |\mu_\alpha-c|^2=R^2
%         \qquad
%         \text{for all }\alpha\in S
% \]
% for some \(R\geq 0\). Then every unit vector \(u\) satisfying
% \[
%         \|P_\alpha u\|^2=p_\alpha \quad (\alpha\in S),
%         \qquad
%         P_\beta u=0 \quad (\beta\notin S),
% \]
% defines a critical point of \(\rho_A\), with
% \[
%         z_A(u)=c,\qquad
%         \rho_A(u)=R^2.
% \]

Finally, for a fixed pair \((S,p)\), the corresponding critical set in the
unit sphere is
$$
        \widetilde C(S,p)
        =
        \left\{
        u\in S^{2n-1}:
        \|P_\alpha u\|^2=p_\alpha\ (\alpha\in S),
        \quad
        P_\beta u=0\ (\beta\notin S)
        \right\}.
$$
% It is naturally diffeomorphic to
% \[
%         \widetilde C(S,p)
%         \cong
%         \prod_{\alpha\in S} S^{2m_\alpha-1}.
% \]
% Its image in projective space is $       C(S,p)=\widetilde C(S,p)/S^1
%         \subset \mathbb CP^{n-1},
% $
% where \(S^1\) acts by simultaneous multiplication of all components by a
% common phase. In particular,
% \[
%         \dim_{\mathbb R} C(S,p)
%         =
%         2\sum_{\alpha\in S}m_\alpha-|S|-1.
% \]
\end{thm}

\begin{proof}
Let $  z=z(u),\,\, B=A-zI.$
A point \([u]\in\mathbb CP^{n-1}\) is critical for \(\rho\) if and only if
$ B^*Bu=\rho(u)u. $
Since \(A\) is normal, the spectral decomposition gives
\[
        B^*B
        =
        \sum_{\alpha=1}^d |\mu_\alpha-z|^2 P_\alpha .
\]
Thus the critical equation is equivalent to
\[
        \sum_{\alpha=1}^d |\mu_\alpha-z|^2 u_\alpha
        =
        \rho_A(u)\sum_{\alpha=1}^d u_\alpha .
\]
Because the eigenspaces \(E_\alpha\) are mutually orthogonal, this holds if
and only if
\[
        \bigl(|\mu_\alpha-z|^2-\rho_A(u)\bigr)u_\alpha=0
        \qquad
        (\alpha=1,\ldots,d).
\]
For every active \(\alpha\), we have \(u_\alpha\neq 0\), and hence $
        |\mu_\alpha-z|^2=\rho_A(u)$.
This proves the critical point criterion.

It remains only to identify the fiber. For each active \(\alpha\), the
condition $
        \|P_\alpha u\|^2=p_\alpha $
says that \(u_\alpha\) lies on the sphere of radius \(\sqrt{p_\alpha}\) in the
complex vector space \(E_\alpha\). This sphere is diffeomorphic to
\(S^{2m_\alpha-1}\). The inactive components are zero. Hence
$$
        \widetilde C(S,p)
        \cong
        \prod_{\alpha\in S}S^{2m_\alpha-1}.
$$
\end{proof}

\begin{rem}
If \(A\) has simple spectrum, then \(d=n\) and \(m_\alpha=1\) for all
\(\alpha\). In this case the active eigenspaces are one-dimensional, and for
fixed admissible weights \(p_i\) the only remaining freedom is the choice of
phases of the active coordinates. Thus, if \(k=|S|\),
       $ \widetilde C(S,p)\cong T^k$
on the unit sphere, and $C(S,p)\cong T^{k-1}$
in projective space.

\end{rem}
We now discuss corollaries for a generic normal matrix. In particular, we assume that the spectrum is simple, and we investigate the two simplest generic configurations of the eigenvalues: two eigenvalues on a segment and three eigenvalues on a circle; see the proof of Theorem~\ref{thm:gen_crit_skel}.
\begin{definition}[General position]
We say that the eigenvalues are in general position if
\begin{enumerate}
    \item no three eigenvalues are collinear;
    \item no four eigenvalues are concyclic;
    \item no triangle formed by three eigenvalues is right-angled.
\end{enumerate}
\end{definition}

\begin{thm}~\label{thm:gen_crit_skel}
Assume $A$ is normal and has pairwise distinct eigenvalues in general position. Then the critical set of $\rho$ on $\mathbb{CP}^{n-1}$ consists exactly of the following components:

\begin{enumerate}
    \item $n$ isolated minima, namely the eigenlines
    \[
        [e_1],\ldots,[e_n];
    \]
    \item for every unordered pair $\{i,j\}$, one critical circle $T^1$ with
    \[
        |u_i|^2=|u_j|^2=\frac12,
        \qquad
        u_k=0\quad(k\neq i,j);
    \]
    thus there are exactly $\binom n2$ pair circles;
    \item for every unordered triple $\{i,j,k\}$ such that $\lambda_i,\lambda_j,\lambda_k$ form an acute triangle, one critical two-torus $T^2$, whose squared moduli are the barycentric coordinates of the circumcenter of the triangle.
\end{enumerate}
There are no other critical points.
\end{thm}

\begin{proof}
Let $[u]$ be critical and let $S$ be its active support.

If $|S|=1$, then $[u]$ is one of the eigenlines. These give the $n$ isolated minima.

If $|S|=2$, using Theorem~\ref{thm:normal-critical-set} we immediately obtain that for every unordered pair $\{i,j\}$, there is exactly one critical probability vector with support $\{i,j\}$, namely
\[
    p_i=p_j=\frac12.
\]
The corresponding critical component in $\mathbb{CP}^{n-1}$ is a circle $T^1$, given by
\[
    |u_i|^2=|u_j|^2=\frac12,
    \qquad
    u_k=0\quad(k\neq i,j),
\]
modulo global phase. Its critical value is
\[
    z=\frac{\lambda_i+\lambda_j}{2},
    \qquad
    \rho_A=\frac{|\lambda_i-\lambda_j|^2}{4}.
\]
On the unit sphere, its preimage is a two-torus $T^2$.

If $|S|=3$, the general-position assumption excludes collinearity. Therefore the three eigenvalues form a nondegenerate triangle. Critical points with active support $S$ exist if and only if the circumcenter $c$ of the triangle with vertices
\[
    \lambda_i,
    \lambda_j,
    \lambda_k
\]
lies in the interior of that triangle. Equivalently, the triangle is acute.

When this happens, the critical probability vector is unique: it is the vector of barycentric coordinates of the circumcenter $c$ with respect to the triangle. The corresponding critical component in $\mathbb{CP}^{n-1}$ is a two-torus $T^2$. Its critical value is
\[
    z=c,
    \qquad
    \rho_A=R^2,
\]
where $R$ is the circumradius. On the unit sphere, its preimage is $T^3$. 

The exclusion of four cocircular eigenvalues rules out active supports
with four or more indices, so these cases exhaust the critical set.
\end{proof}

\section{Escape time for critical sets}\label{sec:quan_prop}

A positive critical component is not an attracting local minimum. However, a trajectory may spend a long finite time near it if the initial point is very close to its stable manifold. Here we give the classical Gronwall-style estimates for a hyperbolic unstable coordinate, see Theorem~\ref{thm:local-trapping}. For clarity, in this section we treat only the case of simple spectrum, though generalization to multiple eigenvalues is at hand. In the normal case the escape exponent is explicit: it is the sum of the positive linear escape rates, and these rates are determined by elementary geometry of the eigenvalues. If Theorem~\ref{thm:gen_crit_skel} showed that a two-dimensional critical component corresponds to an acute triangle formed of some eigenvalues, here we show that it also matters \textit{how exactly} this triangle is acute.

Before proceeding, we recall two standard notions from dynamical systems.
\begin{definition}
Let $\Phi_t$ denote the residual flow. The stable manifold of a critical component $C$ is
\[
    W^s(C)=\{x:\dist(\Phi_t(x),C)\to0\text{ as }t\to+\infty\}.
\]
Locally, for a saddle component, $W^s(C)$ is obtained by setting all unstable coordinates equal to zero.

The unstable dimension $m(C)$ is the number of real transverse directions in which the negative gradient flow moves away from $C$. Equivalently, it is the number of positive eigenvalues of the linearized flow transverse to $C$, counted with multiplicity.
\end{definition}
% All constants below correspond to the following normalization of the negative residual-gradient flow:
% \[
%     \dot u=-2\left((A^*-\overline z I)(A-zI)u-\rho u\right),
%     \qquad
%     z=u^*Au,
%     \qquad
%     \rho=\rho_A(u).
% \]

As before, let
\[
    A=\diag(\lambda_1,\ldots,\lambda_n),
\]
where the eigenvalues $\lambda_i\in\mathbb C$ are pairwise distinct. 

\begin{definition}
    For simple spectrum, a projective critical component $C(p^*)$ with fixed spectral weights $p^*$, as in Theorem~\ref{thm:normal-critical-set}, is called a critical torus.
\end{definition}

Fix a critical torus $C=C(p^*)$. Write
\[
    c=z(p^*),
    \qquad
    R^2=\rho(p^*),
    \qquad
    q_i=\lambda_i-c.
\]
For active indices $i\in S$ we have
\[
    |q_i|=R,
    \qquad
    \sum_{i\in S}p_i^*q_i=0.
\]

There are two kinds of transverse directions: active-support directions and inactive-coordinate directions.
First, we investigate inactive-coordinate directions. \begin{prop}\label{prop:lin_inact} The eigenvalues that lie strictly inside the critical circle, correspond to unstable directions, while those who lie strictly outside the circle correspond to stable directions. Namely,
$$
    |\lambda_k-c|<R
    \quad\Longleftrightarrow\quad
    \text{the inactive coordinate }k\text{ is unstable},
$$
$$
    |\lambda_k-c|>R
    \quad\Longleftrightarrow\quad
    \text{the inactive coordinate }k\text{ is stable}.
$$
\end{prop}
\begin{proof}
    Let $k\notin S$. Near $C$, the coordinate $u_k$ itself is a local complex transverse coordinate. Since the coefficient in the equation for $u_k$ is real,
\[
    \dot u_k
    =
    2\left(\rho-|\lambda_k-z|^2\right)u_k.
\]
At the critical torus this becomes
\[
    \dot u_k
    =
    2\left(R^2-|\lambda_k-c|^2\right)u_k
    +\text{higher order terms}.
\]
Therefore the two real directions corresponding to the complex coordinate $u_k$ have the same linear rate
\[
    \alpha_k=2\left(R^2-|\lambda_k-c|^2\right).
\]
Equivalently,
\[
    \dot p_k
    =
    4\left(R^2-|\lambda_k-c|^2\right)p_k
    +\text{higher order terms}.
\]
Thus, the behavior of the flow is determined by the sign of $R^2-|\lambda_k-c|^2$, and the result follows.
\end{proof}
An active-support perturbation changes the weights $p_i$, $i\in S$, while keeping all inactive coordinates zero. Let
\[
    V_S=\left\{a=(a_i)_{i\in S}\in\mathbb R^S:\sum_{i\in S}a_i=0\right\} \,\,\, \text{and} \,\,\,
    \delta z=\sum_{i\in S}a_i\lambda_i=\sum_{i\in S}a_iq_i,
\]
and write $
    p_i=p_i^*+a_i, \,\, \text{where} \,\,
    \sum_{i\in S}a_i=0 $. Here $V_S$ is the space of variations in \textit{active} variables. 

\begin{prop}\label{prop:lin_act}
The linearized flow in the active variables is
\[
    \boxed{
    (L_Sa)_i
    =
    8p_i^*\,\Rea\left(q_i\overline{\delta z}\right),
    \quad
    i\in S.
    }
\]
Equivalently, if one identifies the complex plane with $\mathbb R^2$ and writes $q_i\cdot y=\Rea(q_i\overline y)$, then
\[
    (L_Sa)_i=8p_i^*(q_i\cdot \delta z).
\]
\end{prop}

\begin{proof}
Let $d_i(p)=|\lambda_i-z(p)|^2$, then
the simplex vector field is
$$
    F_i(p)=4p_i(\rho(p)-d_i(p)).
$$
At the critical point $p^*$, for $i\in S$,
\[
    \rho(p^*)=d_i(p^*)=R^2.
\]
Therefore the linearization in active directions is
\[
    \delta F_i=4p_i^*(\delta\rho-\delta d_i).
\]
Now
\[
    \delta d_i
    =
    \delta |q_i-\delta z|^2\big|_{\delta z=0}
    =
    -2\Rea(q_i\overline{\delta z}).
\]
Also,
\[
\begin{aligned}
    \delta\rho
    &=
    \delta\left(\sum_{i\in S}p_i|\lambda_i-z|^2\right) =
    \sum_{i\in S}a_iR^2+
    \sum_{i\in S}p_i^*\delta d_i.
\end{aligned}
\]
The first term vanishes because $\sum_{i\in S}a_i=0$. The second term is
\[
    \sum_{i\in S}p_i^*(-2\Rea(q_i\overline{\delta z}))
    =
    -2\Rea\left(\left(\sum_{i\in S}p_i^*q_i\right)\overline{\delta z}\right)
    =0,
\]
because $\sum_{i\in S}p_i^*q_i=0$. Hence $\delta\rho=0$, and so
\[
    \delta F_i
    =
    -4p_i^*\delta d_i
    =
    8p_i^*\Rea(q_i\overline{\delta z}).
\]
This proves the formula.
\end{proof}
\begin{rem}

In fact, the active rates  are characterized by the configuration of the eigenvalues on the plane. Define the real symmetric positive semidefinite matrix
\[
    M_S=\sum_{i\in S}p_i^* q_iq_i^T,
\]
where $q_i$ is regarded as a vector in $\mathbb R^2$. Then
\[
    \delta\dot z
    =
    \sum_{i\in S}(L_Sa)_iq_i
    =
    8M_S\delta z.
\]
Thus the nonzero active escape rates are precisely the nonzero eigenvalues of $8M_S$.
\end{rem}

We assume below that no inactive eigenvalue lies on the critical circle:
\[
    |\lambda_k-c|\neq R
    \qquad(k\notin S).
\]
This excludes extra zero directions.

For the critical torus $C=C(p^*)$, define the active positive rates
\[
    \gamma_1^{\rm act},\ldots,\gamma_r^{\rm act}>0
\]
to be the positive eigenvalues of $L_S$ on $V_S$, ignoring zero directions tangent to a possible larger critical family. 

Let
\[
    I_{\rm in}(C)=\{k\notin S: |\lambda_k-c|<R\}
\]
be the set of inactive eigenvalues lying strictly inside the critical circle.

Define
\[
    \boxed{
    m(C)=r+2|I_{\rm in}(C)|
    }
\]
and
\[
    \boxed{
    \Gamma(C)=
    \sum_{\ell=1}^r\gamma_\ell^{\rm act}
    +
    4\sum_{k\in I_{\rm in}(C)}
    \left(R^2-|\lambda_k-c|^2\right).
    }
\]
The number $m(C)$ is the unstable dimension. The number $\Gamma(C)$ is the local escape exponent.

\begin{rem}
In general position, a pair component has unstable dimension
$1+2|I_{\rm in}(C)|$, so its dimension is minimal precisely when its
diametral disk contains no other eigenvalue. This is the empty-disk test
defining a Gabriel edge; see~\cite{GabrielSokal}.
For an acute triple component the unstable dimension is
$2+2|I_{\rm in}(C)|$. It is minimal when the circumdisk is empty, that is,
when the triple is a centered Delaunay triangle;
see~\cite{EdelsbrunnerMesh,BauerEdelsbrunner}.
These are local index tests. Identifying global basin adjacencies would
require additional information about the stable manifolds.
\end{rem}

\begin{thm}\label{thm:local-trapping}
Assume that $C=C(p^*)$ is a Morse--Bott critical torus for the residual flow, that is, the kernel of the Hessian is exactly tangent to $C$. Assume also that no inactive eigenvalue lies on its critical circle:
\[
    |\lambda_k-c|\neq R
    \qquad(k\notin S).
\]
Let $\nu$ denote Fubini-Study volume on $\mathbb{CP}^{n-1}$. There exists a sufficiently small adapted tubular neighborhood $U$ of $C$ such that, if
\[
    E_T(U,C)=\{x\in U:\Phi_t(x)\in U\text{ for all }0\leq t\leq T\},
\]
then there are constants $K_1,K_2>0$ and $T_0>0$ such that for all $T\geq T_0$,
\[
    K_1e^{-\Gamma(C)T}
    \leq
    \nu(E_T(U,C))
    \leq
    K_2e^{-\Gamma(C)T}.
\]
In particular,
\[
    \boxed{
    \lim_{T\to\infty}\frac1T\log \nu(E_T(U,C))=-\Gamma(C).
    }
\]
Thus finite-time trapping near $C$ decays exponentially, and the exponent is the sum of all positive local escape rates.
\end{thm}

\begin{proof}

This is the usual Gronwall estimate for a hyperbolic unstable coordinate: to stay inside a fixed-size neighborhood until time $T$, the initial unstable coordinate must be of size $e^{-\alpha_aT}$. The rates are collected in Propositions~\ref{prop:lin_inact} and \ref{prop:lin_act} and combining them we obtain $\Gamma(C)$.
 The coordinates $\theta$ along $C$ and the stable coordinates $y$ contribute only $T$-independent factors. The unstable coordinates contribute the product
\[
    \prod_{a=1}^m e^{-\alpha_aT}
    =
    e^{-(\alpha_1+\cdots+\alpha_m)T}
    =
    e^{-\Gamma(C)T}.
\]
Thus, for suitable constants $K_1,K_2>0$,
\[
    K_1e^{-\Gamma(C)T}
    \leq
    \nu(E_T(U,C))
    \leq
    K_2e^{-\Gamma(C)T}
\]
for all sufficiently large $T$. Taking logarithms and dividing by $T$ gives the claimed limit.
\end{proof}

\begin{cor}

Assume now that $A$ is Hermitian and
    $$\lambda_1<\lambda_2<\cdots<\lambda_n. $$
Then all eigenvalues lie on a line, and the positive critical components are exactly the pair components $C_{ij}$.

For $i<j$, put
\[
    \Delta_{ij}=\lambda_j-\lambda_i,
    \qquad
    c_{ij}=\frac{\lambda_i+\lambda_j}{2}.
\]
The inactive eigenvalues inside the pair circle are exactly the intermediate eigenvalues
\[
    \lambda_i<\lambda_k<\lambda_j.
\]
Therefore
\[
    \boxed{
    \Gamma(C_{ij})
    =
    2\Delta_{ij}^2
    +
    4\sum_{i<k<j}
    \left(
        \frac{\Delta_{ij}^2}{4}
        -
        (\lambda_k-c_{ij})^2
    \right).
    }
\]
In particular, if $\lambda_i$ and $\lambda_{i+1}$ are adjacent, the sum is empty and
\[
    \boxed{
    \Gamma(C_{i,i+1})=2(\lambda_{i+1}-\lambda_i)^2.
    }
\]
\end{cor}
Thus the slowest pair bottlenecks in the Hermitian case are adjacent small-gap pairs. The characteristic trapping time scale near such a component is
\[
    T_{i,i+1}\asymp \frac{1}{2(\lambda_{i+1}-\lambda_i)^2}.
\]

\section{Polynomial filtering near the critical set}\label{sec:critical-filtering}
It is easy to see when a trajectory is near the saddle set: naturally, the residual is relatively big and changes slowly. It is, therefore, instrumental to accelerate escape from the neighbourhood of the saddle set, and there might be different practical ways to do so. Here we suggest and analyze one of them, perhaps one of the simplest: once a trajectory enters neighbourhood of a critical component, apply a random low-degree polynomial in $A$.

For a normal matrix, the effect of a polynomial filter is especially
simple in spectral coordinates.
\begin{prop}\label{prop:polynomial-reweighting}
Let $q$ be a polynomial such that $q(\lambda_i)\neq0$ for every $i$.
For a unit vector $u$, put
\[
    v=\frac{q(A)u}{\|q(A)u\|},\qquad a_i=|q(\lambda_i)|^2.
\]
Then
\[
    p_i^{(q)}=|v_i|^2=\frac{a_ip_i}{\sum_j a_jp_j}.
\]
\end{prop}
\begin{proof}
In an orthonormal eigenbasis, $(q(A)u)_i=q(\lambda_i)u_i$.
Taking squared moduli and normalizing gives the formula.
\end{proof}
This identity applies to any current state of a trajectory. The distribution
of the filtered ensemble, when $u$ is Haar-distributed, is treated separately
in Section~\ref{sec:concentration}.

Write $a_i=e^{\eta_i}$; adding a common constant to the $\eta_i$ does not
change the filtered vector's spectral weights.
The active linearization is self-adjoint for the inner product
$\langle a,b\rangle_* =\sum_{i\in S}a_ib_i/p_i^*$.
Below, $P_U$ denotes the corresponding orthogonal projection onto its
unstable subspace.
\begin{prop}\label{prop:polynomial-kick}
For small $\eta=(\eta_1,\ldots,\eta_n)$,
\[
    \boxed{
    p_i^{(q)}-p_i
    =
    p_i(\eta_i-\overline\eta_p)+O(\|\eta\|^2),
    }
\]
where
\[
    \overline\eta_p=\sum_jp_j\eta_j.
\]
Consequently, if $U$ denotes the unstable subspace of the active-support linearization in probability coordinates, then the unstable coordinate injected by the filter is
\[
    \boxed{
    P_U\bigl(p^{(q)}-p\bigr)
    =
    P_U\left(p_i(\eta_i-\overline\eta_p)\right)_{i=1}^n
    +O(\|\eta\|^2).
    }
\]
Let $b_q=P_U(p^{(q)}-p)$ and let $\gamma_{\min}$ be the smallest
positive active escape rate. In the linearized active model, the time to
reach an unstable-coordinate threshold $\delta>\|b_q\|$ satisfies
\[
    T_{\rm esc}^{q}\leq
    \frac1{\gamma_{\min}}\log\frac{\delta}{\|b_q\|},
\]
provided $b_q\neq0$, with the norm chosen in the self-adjoint active coordinates.
\end{prop}

\begin{proof}
We have
\[
    p_i^{(q)}=
    \frac{p_ie^{\eta_i}}
    {\sum_jp_je^{\eta_j}}.
\]
For small $\eta$,
\[
    e^{\eta_i}=1+\eta_i+O(\|\eta\|^2)
\]
and
\[
    \sum_jp_je^{\eta_j}
    =
    1+\sum_jp_j\eta_j+O(\|\eta\|^2)
    =
    1+\overline\eta_p+O(\|\eta\|^2).
\]
Therefore
\[
    \frac{1}{\sum_jp_je^{\eta_j}}
    =
    1-\overline\eta_p+O(\|\eta\|^2).
\]
Multiplying gives
\[
    p_i^{(q)}
    =
    p_i(1+\eta_i)(1-\overline\eta_p)+O(\|\eta\|^2)
    =
    p_i+p_i(\eta_i-\overline\eta_p)+O(\|\eta\|^2).
\]
Projecting onto the unstable subspace gives the second formula. In the linearized active model, the unstable norm grows at least as $e^{\gamma_{\min}t}\|b_q\|$, which gives the escape-time bound.
\end{proof}
Consider the face of the simplex supported on two eigenvalues $\lambda_i,\lambda_j$. Put
\[
    \Delta_{ij}=|\lambda_i-\lambda_j|.
\]
Inside this face define
\[
    s=\frac{p_i-p_j}{p_i+p_j}.
\]
Since $p_i+p_j=1$ on the face, this is simply $s=p_i-p_j$.

\begin{prop}\label{prop:pair-escape}
On the two-eigenvalue face, the imbalance satisfies
\[
    \boxed{
    \dot s=2\Delta_{ij}^2s(1-s^2).
    }
\]
Let
\[
    \gamma_{ij}=2\Delta_{ij}^2.
\]
Fix an escape threshold $0<s_*<1$. If $0<|s_0|<s_*$, the time needed to reach $|s(t)|=s_*$ is
\[
    \boxed{
    T_{\rm esc}(s_0)=
    \frac{1}{2\gamma_{ij}}
    \log\left(
    \frac{s_*^2(1-s_0^2)}
    {s_0^2(1-s_*^2)}
    \right).
    }
\]
In particular, as $s_0\to0$,
\[
    \boxed{
    T_{\rm esc}(s_0)=
    \frac{1}{\gamma_{ij}}\log\frac1{|s_0|}+O(1).
    }
\]
\end{prop}

\begin{proof}
On the two-eigenvalue face write
\[
    p_i=\frac{1+s}{2},
    \qquad
    p_j=\frac{1-s}{2}.
\]
Then
\[
    z=p_i\lambda_i+p_j\lambda_j
    =\frac{\lambda_i+\lambda_j}{2}+\frac{s}{2}(\lambda_i-\lambda_j).
\]
A direct computation gives
\[
    \rho=p_i|\lambda_i-z|^2+p_j|\lambda_j-z|^2
    =\frac{\Delta_{ij}^2}{4}(1-s^2).
\]
Moreover
\[
    |\lambda_i-z|^2=\frac{\Delta_{ij}^2}{4}(1-s)^2,
    \qquad
    |\lambda_j-z|^2=\frac{\Delta_{ij}^2}{4}(1+s)^2.
\]
Using
\[
    \dot p_k=4p_k(\rho-|\lambda_k-z|^2),
\]
we find
\[
    \dot p_i
    =4\frac{1+s}{2}\left[
    \frac{\Delta_{ij}^2}{4}(1-s^2)-
    \frac{\Delta_{ij}^2}{4}(1-s)^2
    \right]
    =\Delta_{ij}^2 s(1-s^2).
\]
Since $s=2p_i-1$, we get
\[
    \dot s=2\dot p_i=2\Delta_{ij}^2s(1-s^2).
\]
Thus $y=s^2$ satisfies
\[
    \dot y=2\gamma_{ij}y(1-y).
\]
Solving this logistic equation gives
\[
    \frac{y(t)}{1-y(t)}=
    \frac{y_0}{1-y_0}e^{2\gamma_{ij}t}.
\]
Setting $y(t)=s_*^2$ and $y_0=s_0^2$ gives
\[
    T_{\rm esc}(s_0)=
    \frac{1}{2\gamma_{ij}}
    \log\left(
    \frac{s_*^2(1-s_0^2)}
    {s_0^2(1-s_*^2)}
    \right).
\]
The asymptotic formula follows immediately.
\end{proof}

\begin{cor}[Acceleration by a polynomial filter]\label{cor:pair-acceleration}
On the two-eigenvalue face, suppose that a polynomial filter changes the initial imbalance from $s_0$ to $s_0^{(q)}$, and both satisfy
\[
    0<|s_0|,|s_0^{(q)}|<s_*.
\]
Then the exact time saved is
\[
    \boxed{
    T_{\rm esc}(s_0)-T_{\rm esc}(s_0^{(q)})
    =
    \frac{1}{2\gamma_{ij}}
    \log\left(
    \frac{(s_0^{(q)})^2(1-s_0^2)}
    {s_0^2(1-(s_0^{(q)})^2)}
    \right).
    }
\]
If both imbalances are small, then
\[
    \boxed{
    T_{\rm esc}(s_0)-T_{\rm esc}(s_0^{(q)})
    =
    \frac{1}{\gamma_{ij}}
    \log\frac{|s_0^{(q)}|}{|s_0|}+o(1).
    }
\]
In particular, doubling the unstable imbalance saves approximately
\[
    \boxed{
    \frac{\log 2}{2|\lambda_i-\lambda_j|^2}
    }
\]
units of residual-flow time.
\end{cor}

\begin{proof}
Subtract the exact escape-time formula for $s_0^{(q)}$ from the exact escape-time formula for $s_0$. Since $\gamma_{ij}=2|\lambda_i-\lambda_j|^2$, the final sentence follows from the small-imbalance asymptotic.
\end{proof}

It remains to analyze how polynomial filtering acts on the imbalance coordinate.

Fix two indices $i,j$ and assume $p_i,p_j>0$. Define
\[
    \ell_{ij}=\log\frac{p_i}{p_j}.
\]
Then
\[
    s_{ij}=\tanh\left(\frac{\ell_{ij}}2\right).
\]

\section{Concentration phenomena and polynomial filtering}
\label{sec:concentration}
The preceding section concerns escape from a neighborhood of a critical
component. We now turn to a different question: how the initial measure
determines the part of the spectrum explored by an ensemble of trajectories.
Throughout this section $A$ is normal, and we work in an orthonormal
eigenbasis.

Consider a matrix with $120$ eigenvalues sampled independently and uniformly
by area from a disk, then translated to have mean zero and scaled to spectral
diameter $D=2$. We evolve $600$ Haar-distributed starts by the retracted
Euler flow of Proposition~\ref{prop:gradient-flow}, with step $h=0.0125$.
Write $\tau=D^2t$ for dimensionless time.
Figure~\ref{fig:paper-haar-concentration} compares the initial and final
Rayleigh clouds at $\tau=0$ and $\tau=30$. The eigenvalues occupy the full
disk, but the initial Rayleigh points concentrate near its center, and the
final cloud still explores mainly the central part of the spectrum.
The median modulus increases from $\PaperDiskHaarInitialRadius$ to
$\PaperDiskHaarFinalRadius$.

\begin{figure}[H]
    \centering
    \includegraphics[width=0.92\linewidth]{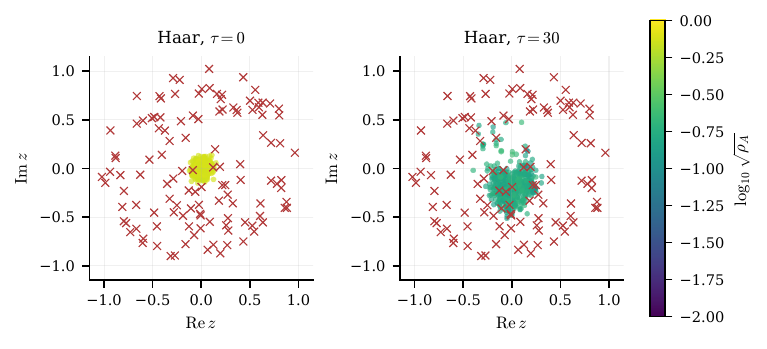}
    \caption{Concentration for a normal matrix with a random disk spectrum.
    Left: the initial Rayleigh points of $600$ Haar starts. Right: the same
    trajectories at $\tau=30$. Red crosses mark the $120$ eigenvalues, and
    colors encode the residual norm on a common logarithmic scale. No
    polynomial filter is applied; both panels have the same axis limits.}
    \label{fig:paper-haar-concentration}
\end{figure}

The concentration has a simple quantitative origin. For a Haar-distributed
unit vector $u\in\mathbb C^n$, the weights $p_i=|u_i|^2$ are uniformly
distributed on the simplex. Equivalently,
\[
    p_i=\frac{E_i}{\sum_jE_j},
\]
where the $E_i$ are independent exponential random variables of mean one.
The distribution of $z_A(u)$ is the numerical shadow of $A$;
see~\cite{DUNKL20112042}.

\begin{prop}[Haar concentration]\label{prop:haar-concentration}
Put $\bar\lambda=n^{-1}\sum_i\lambda_i$. Then
\[
    \mathbb E z_A(u)=\bar\lambda,\qquad
    \mathbb E|z_A(u)-\bar\lambda|^2
    =\frac{1}{n+1}\left(\frac1n\sum_i|\lambda_i-\bar\lambda|^2\right).
\]
In particular, if $s^2=n^{-1}\sum_i|\lambda_i-\bar\lambda|^2$, then
for every $r>0$,
\[
    \mathbb P\{|z_A(u)-\bar\lambda|\geq r\}
    \leq\frac{s^2}{(n+1)r^2}.
\]
\end{prop}
\begin{proof}
The Dirichlet moments are
$\mathbb E p_i=1/n$ and
$\mathbb E p_ip_j=(1+\delta_{ij})/[n(n+1)]$.
Using $z_A(u)=\sum_i p_i\lambda_i$ gives the mean. Since
$\sum_i(\lambda_i-\bar\lambda)=0$, the mixed terms in the centered second
moment cancel, giving the stated variance. The probability bound follows
by applying Markov's inequality to $|z_A(u)-\bar\lambda|^2$.
\end{proof}
Thus, for spectra of bounded diameter, the root-mean-square width of the
initial Rayleigh cloud is of order $n^{-1/2}$. This concerns the initial
ensemble; the probabilities of eventual capture also depend on the basins
of the nonlinear flow. It nevertheless motivates changing the initial
measure before applying that flow.

Let $q$ be a polynomial with $a_i=|q(\lambda_i)|^2>0$.
By Proposition~\ref{prop:polynomial-reweighting}, the filtered Haar weights
are
\[
    p_i^{(q)}=\frac{a_iE_i}{\sum_j a_jE_j}.
\]
\begin{prop}\label{prop:filtered-density}
With respect to $dp_1\cdots dp_{n-1}$ on the simplex, the density is
\[
    f_q(p)=(n-1)!\,
    \frac{\prod_i a_i^{-1}}{\left(\sum_i p_i/a_i\right)^n}.
\]
\end{prop}
\begin{proof}
Put $Y_i=a_iE_i$. The joint density of $Y$ is
$\prod_i a_i^{-1}\exp(-\sum_iY_i/a_i)$.
Writing $Y_i=sp_i$, where $s>0$, gives the Jacobian $s^{n-1}$.
Integrating out $s$ yields
\[
    f_q(p)=\prod_i a_i^{-1}\int_0^\infty
    e^{-s\sum_i p_i/a_i}s^{n-1}\,ds
    =(n-1)!\,\frac{\prod_i a_i^{-1}}
    {\left(\sum_i p_i/a_i\right)^n}.
\]
\end{proof}
\begin{cor}\label{cor:filtered-regions}
Let $B_1,B_2\subset\Delta^{n-1}$ be measurable sets of positive volume, and let $h_1,h_2>0$. Suppose that
\[
    \sum_i p_i/a_i\le h_1
    \qquad\text{for }p\in B_1,
\]
and
\[
    \sum_i p_i/a_i\ge h_2
    \qquad\text{for }p\in B_2.
\]
Then
\[
    \boxed{
    \frac{\mathbb P_q(B_1)}{\mathbb P_q(B_2)}
    \ge
    \left(\frac{h_2}{h_1}\right)^n
    \frac{\Vol(B_1)}{\Vol(B_2)}.
    }
\]
\end{cor}

\begin{proof}
By the density formula,
\[
    f_q(p)\ge C h_1^{-n}\quad\text{on }B_1,
    \qquad
    f_q(p)\le C h_2^{-n}\quad\text{on }B_2,
\]
where
\[
    C=(n-1)!\prod_i a_i^{-1}.
\]
Integrating gives
\[
    \mathbb P_q(B_1)\ge C h_1^{-n}\Vol(B_1),
    \qquad
    \mathbb P_q(B_2)\le C h_2^{-n}\Vol(B_2),
\]
which implies the claim.
\end{proof}

\begin{rem}
This corollary explains why even a low-degree filter can have a large effect in high dimension. The filter only changes the weights by the factors $a_i=|q(\lambda_i)|^2$, but the induced density contains the $n$-th power
\[
    \left(\sum_i p_i/a_i\right)^{-n}.
\]
Thus a modest spectral tilt can significantly change the probability of starting in one basin rather than another.
\end{rem}

We return to the preceding disk spectrum and the same $600$ underlying
Haar vectors. Independently for each trajectory, draw coefficients
$\xi_k\sim\mathcal N_{\mathbb C}(0,1)$ and set
\[
  q_m(\zeta)=\sum_{k=0}^m\xi_k\zeta^k,\qquad
  v_0=\frac{q_m(A/R)u_0}{\|q_m(A/R)u_0\|},\qquad m\in\{1,4\}.
\]
Here $R=\max_i|\lambda_i|$ is known from the constructed test matrix.
Conditional on a drawn polynomial, Proposition~\ref{prop:filtered-density}
 gives its sampling density; averaging over the polynomials gives the
 random-filter ensemble. The residual flow and its time horizon are unchanged.
Figure~\ref{fig:paper-filtered-concentration} compares the resulting initial
and final ensembles. Even degree one substantially broadens the explored
spectral region.

To quantify the improvement, define finite-time geometric coverage by
\[
  C_\varepsilon(Z_T)=\frac1n\#\{i:\min_{z\in Z_T}|\lambda_i-z|\leq\varepsilon\}.
\]
For $\varepsilon=0.05$, the coverages are
$\PaperDiskHaarCoverage\%$, $\PaperDiskOneCoverage\%$, and
$\PaperDiskFourCoverage\%$ for Haar, degree one, and degree four,
respectively, with no increase in the number of trajectories.

\begin{center}
\small
\begin{tabular}{lrrr}
\toprule
Starts & Geometric coverage & Certified coverage & Median $\sqrt{\rho_A}$\\
\midrule
Haar & $\PaperDiskHaarCoverage\%$ & $\PaperDiskHaarCertifiedCoverage\%$ & $\PaperDiskHaarResidual$\\
Degree $1$ & $\PaperDiskOneCoverage\%$ & $\PaperDiskOneCertifiedCoverage\%$ & $\PaperDiskOneResidual$\\
Degree $4$ & $\PaperDiskFourCoverage\%$ & $\PaperDiskFourCertifiedCoverage\%$ & $\PaperDiskFourResidual$\\
\bottomrule
\end{tabular}
\end{center}

Here certified coverage uses only outputs satisfying
$\sqrt{\rho_A}\leq0.05$, with the same coverage radius $0.05$.
The residual acceptance test is more stringent than the planar-distance
test; the distinction is discussed below.

\begin{figure}[H]
  \centering
  \includegraphics[width=\linewidth]{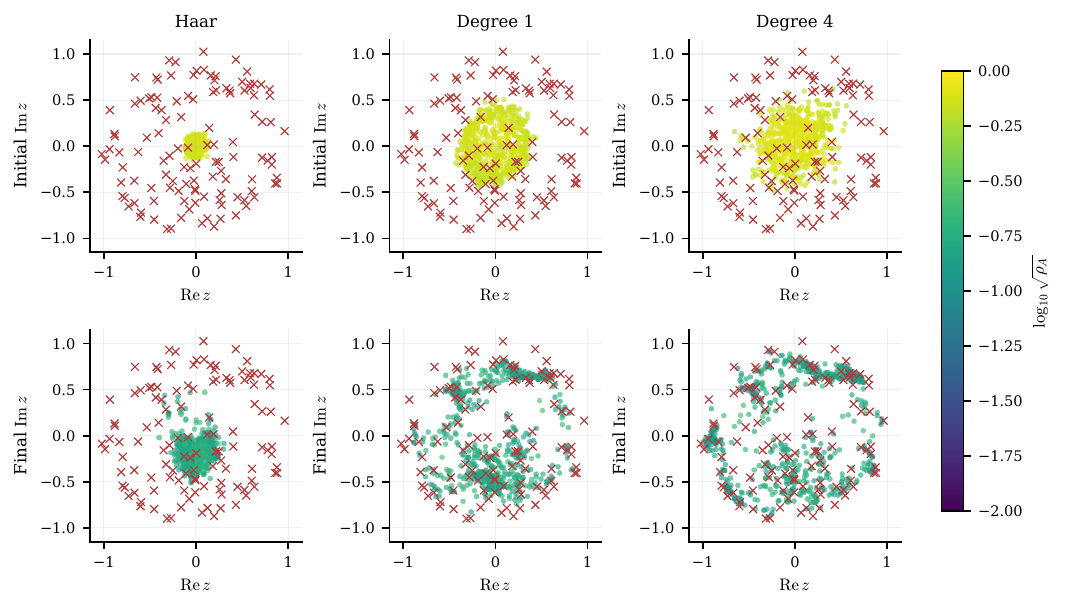}
  \caption{Matched random polynomial filtering for the spectrum in
  Figure~\ref{fig:paper-haar-concentration}. Columns: Haar starts,
  independent degree-one filters, and independent degree-four filters.
  Top: initial states. Bottom: final states at $\tau=30$. All panels
  have the same eigenvalues, axis limits, and residual-color scale.
  Every protocol evolves $600$ trajectories by the same Euler flow;
  filtering costs an additional
  $m$ applications of $A$ per start when evaluated by Horner's rule.}
  \label{fig:paper-filtered-concentration}
\end{figure}

The difference between geometric and certified coverage has a simple
local explanation. Let $A$ be normal with simple spectrum, and let $e_0$
be a unit eigenvector with eigenvalue $\lambda_0$. In an orthonormal
eigenbasis, put $p_i=|u_i|^2$ and $q=1-p_0$. For this fixed spectrum,
as $q\to0$, the weighted-average formulas for $z$ and $\rho_A$ give
\[
    \sin\angle(u,e_0)=\sqrt q,\qquad
    |z-\lambda_0|=O(q),\qquad
    \sqrt{\rho_A}\asymp\sqrt q.
\]
Thus the Rayleigh error is quadratic in the eigenvector error, whereas
the residual norm is first order. Here the comparison is with
$\sqrt{\rho_A}$; the squared residual itself is of order $q$.

In the two-point case, with weights $1-q$ and $q$ on $\lambda_0$ and
$\lambda_1$, put $g=|\lambda_1-\lambda_0|$. Then exactly
\[
    |z-\lambda_0|=gq,\qquad
    \sqrt{\rho_A}=g\sqrt{q(1-q)},\qquad
    \dot q=-4g^2q(1-q)(1-2q).
\]
The last identity follows from
Proposition~\ref{prop:normal-simplex-flow}. For $0<q(0)<1/2$, it gives
$q(t)\sim Ce^{-4g^2t}$ for some $C>0$: the Rayleigh error decays with
rate $4g^2$, while the residual norm and eigenvector error decay with
rate $2g^2$.

This explains why the Rayleigh points in
Figure~\ref{fig:paper-filtered-concentration} can already be close to
the spectrum while their vectors fail the residual acceptance test.
Once trajectories have entered their final capture neighborhoods,
waiting for eigenvector convergence can leave almost the same spatial
picture: further integration sharpens the points toward their limiting
eigenvalues, with smaller residuals and correspondingly different colors
and certified coverage.

\setlength{\bibsep}{4pt plus 1pt minus 1pt}
\bibliographystyle{amsplain}
\bibliography{\PaperBibliography}
\end{document}